\documentclass{amsart}
\usepackage{amssymb,latexsym}

\newcommand{\Q}{\mathbb{Q}}
\newcommand{\Z}{\mathbb{Z}}
\newcommand{\QQ}{\mathcal{Q}}

\newcommand{\E}{\mathcal{E}}
\newcommand{\F}{\mathcal{F}}
\newcommand{\I}{\mathcal{I}}

\newcommand{\calS}{\mathcal{S}}

\newcommand{\cd}{{\bf cd}}
\newcommand{\ctwod}{{\bf c}-2{\bf d}}
\renewcommand{\c}{{\bf c}}
\renewcommand{\d}{{\bf d}}

\def\FAA{\Q \kern .05em \langle y_1,y_2,\dots \rangle}
\def\NC{\Z \kern .05em \langle y_1,y_2,\dots \rangle}

\theoremstyle{plain}
\newtheorem{corollary}{Corollary}[section]

\newtheorem{theorem}{Theorem}[section]
\newtheorem{proposition}{Proposition}[section]
\newtheorem{definition}{Definition}[section]

\theoremstyle{remark}
\newtheorem{remark}{Remark}[section]
\newtheorem{example}{Example}[section]

\newcommand{\vanish}[1]{}

\begin{document}
\title[Quasisymmetric Functions and Enumeration]
{Peak Quasisymmetric Functions\\ and Eulerian Enumeration}
\author{Louis J.\ Billera}
\address{Department of Mathematics, Cornell University, Ithaca, NY 14853-4201}
\thanks{The first two authors were supported in part by NSF grants DMS-9800910
and DMS-0100323.  The third author was supported in part by the Leverhulme Trust.}
\author{Samuel K.\ Hsiao}
\author{Stephanie van Willigenburg}

\begin{abstract}
Via duality of Hopf algebras, there is a direct association between
peak quasisymmetric functions and enumeration of chains in Eulerian posets.
We study this association
explicitly, showing that the notion of $\cd$-index, long studied in the context
of convex polytopes and Eulerian posets, arises as the dual basis
to a natural basis of peak quasisymmetric functions introduced by Stembridge.
Thus Eulerian
posets having a nonnegative $\cd$-index (for example, face lattices of convex
polytopes) correspond to peak quasisymmetric functions having a nonnegative
representation in terms of this basis.  We diagonalize the operator
that associates the basis of descent sets for all quasisymmetric functions
to that of peak sets for the algebra of peak functions, and study the
$g$-polynomial for Eulerian posets as an algebra homomorphism.
\end{abstract}

\maketitle

\section{Introduction \label{intro}}

In the enumerative theory of partially ordered sets, one is often interested
in enumerative functionals that are nonnegative for a given class of posets.
Thus, for example, the {\em generalized lower bound theorem} for convex
polytopes asserts that certain functionals of the flag $f$-vector, the
so-called $g$-vector, will be nonnegative for all convex polytopes.

In recent years, there have been a number of papers linking the enumerative
theory of posets to the study of coalgebras and Hopf algebras, leading to
a deeper understanding of one such funtional, the $\cd$-index of Eulerian posets.
See \cite{Ag,BE,BER,BER2,Ehr,ER} for a sample of such work and \cite{BiBj}
for a relatively recent survey of the state of such enumerative questions.

In the theory of symmetric functions, one is often interested to know when
certain symmetric functions can be expressed as nonnegative linear combinations
of a preferred basis (the Schur functions, for example).  The recent
breakthrough of Haiman on the Macdonald positivity conjecture
\cite{Hai} is one such instance.

Setting questions in posets and symmetric functions in the context of
Hopf algebras has led to a deep understanding of their relationship.
In \cite{GKL}, Gel'fand, {\it et al.}, show the
Hopf algebra of quasisymmetric functions to be dual to the
the Hopf algebra $NC=\NC$, which they called {\em noncommutative symmetric
functions}.  Billera and Liu \cite{BL} considered elements of the algebra
$\FAA = \Q \otimes NC$ as flag-enumeration functionals on all graded
posets, and they defined a quotient $A_\E$ of $\FAA$, which consists of
all such functionals on Eulerian posets.
Bergeron, Mykytiuk, Sottile and van Willigenburg \cite{BMSV,BMSV2} showed
that the algebra $A_\E$ is dual to Stembridge's
algebra $\Pi$ of peak quasisymmetric functions \cite{Stem}.  More precisely,
they showed that both of these algebras have natural coproducts that make
them into Hopf algebras, and that these Hopf algebras are, in fact, dual.
This duality links the study of the enumerative properties of Eulerian posets,
including associated geometric objects such as convex polytopes and
hyperplane arrangements, with that of Stembridge's enriched $P$-partitions and
related questions having to do with peaks and shuffles in permutations.

We will explore some of these links here.  In particular, we will show
that the natural nonnegativity questions on each side are closely related.
The weight enumerators of all enriched $P$-partitions of chains
were shown by Stembridge to be a basis for the peak algebra $\Pi$.
An immediate consequence of the result of Bergeron,
{\it et al.}, is the fact that the formal quasisymmetric function $F(P)$ of an
Eulerian poset $P$, as defined by Ehrenborg \cite{Ehr}, is an element of
$\Pi$.  The coefficients of $F(P)$ in terms of this basis are given by
the $\cd$-index of $P$.  Thus nonnegative
representation for quasisymmetric functions of Eulerian posets is equivalent
to their having a nonnegative $\cd$-index.
More precisely, we show that the linear forms defining the coefficients
of the \ctwod-index give a basis for $A_\E$ dual to Stembridge's basis
for $\Pi$.  This completely unexpected result shows the
$\cd$-index to be a natural concept in spite of its initial {\it ad hoc}
definition.

We give the basic definitions in the rest of \S \ref{intro}.
In \S \ref{qsymmpeak} we define the algebra $\QQ$ of quasisymmetric
functions over $\Q$ and the subalgebra $\Pi$ of peak functions.  In
\S \ref{euler} we discuss graded and Eulerian posets and the algebras of
flag-enumeration functionals on each class.  In \S \ref{hopf} we define
the relevant coproducts on these algebras that make them pairs of dual
Hopf algebras.  Finally, in \S \ref{basissect}, we look at different
bases for $\QQ$ and corresponding representations.

In \S \ref{cdpeak}, we relate the representation of the quasisymmetric function
$F(P)$ in terms of Stembridge's
basis to the $\cd$-index of the poset $P$, in particular to the
\ctwod-index studied in \cite{BER}.  One consequence is that the quasisymmetric
functions corresponding to zonotopes lie in the (half) integral sublattice
of $\Pi$ spanned by the Stembridge basis.

In \S \ref{sigmasect} we consider the map $\vartheta$,
defined and studied by Stembridge,
associating the weight enumerator of all $P$-partitions for
a fixed labeled poset with that of the corresponding enriched $P$-partitions
for the same data.  When applied to a quasisymmetric function coming from a
representable geometric lattice, one obtains the quasisymmetric function
arising from the corresponding zonotope.  We show this map to be diagonalizable
on $\Pi$, and we give an explicit basis of eigenvectors.  The principal
eigenvector in any degree is given by the distribution of peak sets in the
corresponding symmetric group.  In fact, the operator
$\frac{1}{2^{n+1}}\vartheta$ can be viewed a giving a random walk
on the peak sets of $S_{n+1}$ having this stationary distribution.

Finally, in \S\ref{gsect}, we extend the usual
$g$-polynomial of Eulerian posets to the algebra $\Pi$ (in fact to $\QQ$),
where it defines an algebra homomorphism to the polynomial ring $\Q[x]$.
It is hoped that this way of viewing the $g$-invariant will lead to a better
understanding of its properties.

\subsection{Quasisymmetric functions and the peak algebra\label{qsymmpeak}}

We let $\QQ$ denote the algebra of
quasisymmetric functions over $\Q$, that is, all 
bounded degree formal power series
$F$ in variables $x_1,x_2,\dots$ such that for all $m$, and any
$i_1<i_2<\cdots < i_m$, the coefficient of $x_{i_1}x_{i_2}\cdots x_{i_m}$
in $F$ is the same as that of $x_1x_2\cdots x_m$.  Equivalently, $\QQ$
is the linear span of $M_0=1$ and all power series $M_\beta$, where
$\beta = (\beta_1,\beta_2,\dots,\beta_k)$ is a vector of positive
integers (a {\it composition} of $\beta_1+\beta_2+\cdots +\beta_k$) and
\begin{eqnarray}
M_\beta = \sum_{i_1<i_2<\cdots < i_k} x_{i_1}^{\beta_1}
x_{i_2}^{\beta_2}\cdots x_{i_k}^{\beta_k}. 
\label{mbeta}
\end{eqnarray}

We denote by $\QQ_{n+1}$ the subspace of $\QQ$ consisting of those
quasisymmetric functions that are homogeneous of degree $n+1$; equivalently,
$\QQ_{n+1}$ is the linear span of all $M_\beta$, where $\beta$ is a
composition of $n+1$.
It is straightforward to see that the $2^{n}$ such $M_\beta$
form a basis for the vector space $\QQ_{n+1}$.  For integer $k>0$,
let $[k]:=\{1,2,\dots,k\}$ and $[0]=\emptyset$. 
It will be helpful for us to consider the equivalent indexing of this
basis by subsets of $[n]$, where for
$S =\{i_1,i_2,\dots,i_k\} \subset [n]$,
$M_S:=M_{\beta(S)}$ and $\beta(S)=(i_1,i_2-i_1,\dots,i_k-i_{k-1},n+1-i_k)$.
When $n+1$ is not clear from the context, we will write $M_S^{(n+1)}$.
For further details about quasisymmetric functions, see \cite{ECII}.

\begin{definition}
Let $n\ge 0$ and $S\subset [n]$.
\begin{enumerate}
\item
$S$ is said to be {\em left sparse} if $1\notin S$
and $i\in S$ implies $i-1 \notin S$.  
\item
Similarly, $S$ is 
{\em right sparse} if $n\notin S$ and $i\in S$ implies $i+1\notin S$.
\item
For an integer $k$,
let $S+k = \{i+k \ | \ i\in S \}$.
\end{enumerate}
\end{definition}

We note that \cite{Read} uses the terms left and right sparse in the
opposite sense than used here.

The {\it peak algebra} $\Pi$ is defined to be the subalgebra of $\QQ$ generated
by the elements
\begin{eqnarray}
\Theta_S = \sum_{T:S\subset T \cup(T+1)} 2^{|T|+1}M_T,
\label{theta}
\end{eqnarray}
where $S$ is a left sparse subset of $[n]$, $n\ge 0$.
Here the sum is over $T\subset [n]$ and $M_T = M_T^{(n+1)}$.  
Defining $\Pi_n = \Pi \cap \QQ_n$, we have that $\dim_\Q(\Pi_n)=a_n$,
the $n^{th}$ Fibonacci number (indexed so that $a_1=a_2=1$)
\cite[Theorem 3.1]{Stem}.

We consider an equivalent indexing of the  basis of $\Pi$
to that by left sparse subsets in (\ref{theta}).
Let \c~  and \d~  be indeterminates, of degree 1 and 2, respectively.
For a $\cd$-word $w = \c^{n_1}\d\c^{n_2}\d\cdots\c^{n_k}\d\c^{m}$
of degree $n$, define the subset $S_w \subset [n]$ by
\begin{eqnarray*}
S_w &=& \{ n_1+2,n_1+n_2+4,\dots,n_1+n_2+\cdots+n_k+2k \} \\
&=& \{i_1,i_2,\dots,i_k\},
\label{Sw}
\end{eqnarray*}
where $i_j = \deg(\c^{n_1}\d\c^{n_2}\d\cdots\c^{n_j}\d)$.  Note that
$S_w$ is always left sparse and every left sparse $S\subset[n]$ is
of the form $S_w$ for some $\cd$-word $w$ of degree $n$.  Thus, there
will be no ambiguity if we relabel this basis to
\begin{eqnarray}
\Theta_w = \Theta_{S_w},
\label{thetaw}
\end{eqnarray}
where $w$ ranges over all possible $\cd$-words.
(For $w={\bf 1}$, the empty word, we have $\Theta_{\bf 1}=2M^1_\emptyset$.)
Note that
$\deg(\Theta_w)=\deg w + 1$,
so the ambiguity about the degree in the earlier notation is no longer an issue.

\subsection{Eulerian posets and enumeration algebras\label{euler}}
Recall that a {\it graded poset} $P$ is one having a unique minimal element
$\widehat 0$ and maximal element $\widehat 1$ for which every maximal
chain has the same number of elements.  Thus if $x\in P$ has
a maximal chain
$$\widehat 0=x_0<x_1<\cdots<x_k=x,$$
we say that $x$ has {\it rank} $k$, denoted $r(x)=k$ (and
so $r(\widehat 0) = 0$).  Further,
we define the rank of $P$ to be $r(P):=r(\widehat 1)$.
For a graded poset P of rank ${n+1}$
and a subset $S\subset [n]$, we denote by $f_S(P)$
the number of flags ({\it i.e.}, chains) in $P$ having elements with
precisely the ranks in $S$.  Note that the ranks $0$ and $n+1$ are not
included here.  The function $S \mapsto f_S(P)$ is known as the {\it flag
$f$-vector} of $P$.
Recall that a graded poset is said to be {\it Eulerian} if its M\"obius
function  $\mu$ satisfies $\mu(x,y)=(-1)^{r(y)-r(x)}$ for every pair $x\le y$.
See \cite{ECI} for general background in this area.

In \cite{BL}, elements of the free associative algebra $\FAA$ were
associated to flag numbers of graded posets.  If
$\beta = (\beta_1,\beta_2,\dots,\beta_k)$ is a composition of $n+1$,
let $y_\beta = y_{\beta_1} y_{\beta_2}\cdots y_{\beta_k}\in \FAA$.  We
associate $f_S$ for posets of rank $n+1$ to $y_{\beta(S)}$.
For $k\ge 1$, we define
$$\chi_{k} := \sum_{i+j=k}(-1)^{i}y_iy_j,$$
where the sum is over all $i,j \ge 0$ and we set $y_0=1$ for
convenience.  The element $\chi_k$ corresponds to the Euler relation
for rank $k$ posets.
Let $I_\E$ be the two-sided ideal in $\FAA$ generated
by the $\chi_k$, $k\ge 1$, and define the {\it algebra of
forms on Eulerian posets} to be $A_\E = \FAA/I_\E$.  Letting
$\deg(y_i) = i$, $\FAA$ is a graded algebra and, since $I_\E$ is a
homogeneous ideal, so is $A_\E$.  It is shown in \cite{BL} that
\begin{eqnarray}
A_{\E} \cong {\Q} \langle y_1, y_3, y_5, \dots  , y_{2k+1}, \dots  
\rangle
\label{iso2}
\end{eqnarray}
as graded $\Q$ algebras.  As a result, we have that
for $n \ge 1$, the dimension of  $(A_\E)_{n}$ is again $a_{n}$,
the $n$-th Fibonacci number.

\subsection{Coproducts and graded Hopf duality\label{hopf}}
Noting the equality of the dimensions of $\Pi_n$ and $(A_\E)_n$,
Bergeron, {\it et al.}, studied the relationship between them.  To do so,
they described coproducts on $\Pi$ and $A_\E$, respectively, that make
each a Hopf algebra  \cite{BMSV}.  The coproduct on
the subalgebra $\Pi$ is inherited from the usual coproduct on $\QQ$,
defined by $\Delta(M_0)=M_0\otimes M_0$ and
$$\Delta(M_\alpha) = \sum_{\alpha=\alpha_1\cdot \alpha_2} M_{\alpha_1} \otimes
M_{\alpha_2},$$
where $\alpha_1\cdot \alpha_2$ is the concatenation of compositions
$\alpha_1$ and $\alpha_2$, and either $\alpha_1$ or $\alpha_2$ may be the
empty composition of 0.  It was shown in \cite[Theorem 2.2]{BMSV2}
that $\Pi$ is closed under this coproduct.

There is a coproduct on $\FAA$ defined by
\begin{eqnarray}
\Delta(y_k) = \sum_{i+j=k}y_i \otimes y_j,
\label{faadelta}
\end{eqnarray}
where the sum is over all $i,j \ge 0$, which extends to all of
$\FAA$ by virtue of its being an algebra map.  In \cite{BMSV}, it
is shown that this coproduct is well-defined on the quotient $A_\E$.

With the augmentation map that is zero in positive degree and the identity
in degree 0, both $\QQ$ and $\FAA$ are bialgebras.
The existence of an antipode on each of these bialgebras, making them
Hopf algebras, follows from the fact that they are graded
(see, {\it e.g.},\cite[Lemma 2.1]{Ehr}).
More precisely, if $X$ has degree $n$, then
$$\Delta (X) = X\otimes 1 + \sum_{i=0} ^{n-1} Y_i\otimes Z_{n-i} ,$$
where $Y_j$ and $Z_j$ have degree $j$, and the antipode is defined
recursively by $s(1)=1$ and
\begin{eqnarray} 
s(X) = - \sum_{i=0} ^{n-1} s(Y_i)Z_{n-i}.
\label{antirec}
\end{eqnarray}

We can compute the antipode explicitly for $\Pi$ and $A_\E$. 
If we denote by $w^*$ the reverse of the $\cd$-word $w$,
{\it e.g.}, $(\c\c\d)^* = \d\c\c$,
then we have the following.
We delay the proof until \S\ref{thetareps}.

\begin{proposition} In terms of the basis $\{\Theta_w \}$,
the antipode of $\Pi$ is given by
$$s(\Theta _w) = (-1)^{\deg w+1}\Theta _{w^*}.$$
\label{pianti}
\end{proposition}


Recall that if $\beta = (\beta_1,\dots,\beta_k)$ is a composition of $n+1$,
then we denote $y_\beta = y_{\beta_1}\cdots y_{\beta_k} \in\FAA$.  If
$\beta^* = (\beta_k,\dots,\beta_1)$ is the reverse composition, then we have

\begin{proposition} In terms of the basis $\{y_\beta\}$,
the antipode of $A_\E$ is given by
$$s(y_\beta) =  (-1)^{n+1} y_{\beta^*},$$
where $\beta$ is a composition of $n+1$.
\end{proposition}

\begin{proof}
We show first that $s(y_n) = (-1)^n y_n$.
By (\ref{faadelta}) and (\ref{antirec}) we have $s(y_1)=-y_1$. By induction,
$$s(y_n)=- \left(   \sum _{i=0}^{n-1} (-1)^i y_iy_{n-i}\right) 
= - \left( \chi _n - (-1)^n y_n \right).$$
The assertion follows since $\chi _n$ vanishes in $A_\E$.

The proposition now follows from the fact that the antipode is an algebra
anti-homomorphism.
\end{proof}

A key result for our purposes is that $A_\E$ and
$\Pi$ are dual as graded Hopf algebras \cite[Theorem 5.4]{BMSV}.  By dual
we will always mean {\em graded dual}; that is, if a graded algebra
is of the form $V=V_0\oplus V_1 \oplus \cdots$ as a graded vector space,
then its graded dual is, as a vector space, $V^*=V^*_0\oplus V^*_1 \oplus \cdots$,
where $V^*_i$ is the usual dual space to the finite dimensional space $V_i$.

Thus,
we have that elements of $\Pi$ are characterized by having coefficients
that satisfy the generalized Dehn-Sommerville equations for Eulerian posets
\cite{BB}.

\begin{proposition}
If $F=\sum_{S\subset [n]} f_S \thinspace M_S \in \QQ_{n+1}$, then $F\in \Pi$
if and only if $\sum_{S\subset [n]} a_S \thinspace f_S = 0$ whenever
$\sum_{S\subset [n]} a_S \thinspace y_{\beta(S)} \in I_\E$. 
\label{DSQ}
\end{proposition}

If $P$ is any graded poset of rank $n+1$, then following \cite{Ehr} we
define the formal quasisymmetric function associated to $P$ by
\begin{eqnarray}
F(P) = \sum_{S\subset [n]} f_S(P)\thinspace M_S \in \QQ_{n+1}.
\label{EQS}
\end{eqnarray}
Then it follows from Proposition \ref{DSQ} that the quasisymmetric
functions of Eulerian posets are elements of $\Pi$.  However, the converse
does not hold; it is possible for a graded poset $P$ not to be Eulerian,
yet still satisfy $F(P) \in \Pi$.  The smallest such example has
$f_\emptyset=1$, $f_1=f_2=3$ and $f_{12}=6$.

\subsection{Bases and interval representations\label{basissect}}

It will be helpful to consider two other bases for $\QQ$ and the
corresponding representations of arbitrary $F\in \QQ$.
For $S\subset [n]$, we define
\begin{eqnarray}
F_S = \sum_{T \supset S} M_T
\end{eqnarray}
and
\begin{eqnarray}
K_S = \sum_{T \supset S} F_T = \sum_{T \supset S}2^{|T|-|S|} M_T.
\label{KS}
\end{eqnarray}
Again all sums are over $T\subset [n]$ and $M_T = M_T^{(n+1)}$; when the
context does not make it clear we will write $F_S^{(n+1)}$ and $K_S^{(n+1)}$.
It is
easy to check that the $F_S$ and $K_S$ are again bases for $\QQ_{n+1}$
and that 
\begin{eqnarray}
M_S = \sum_{T \supset S}(-1)^{|T|-|S|} \thinspace F_T
\end{eqnarray}
and
\begin{eqnarray}
F_S = \sum_{T \supset S} (-1)^{|T|-|S|} \thinspace  K_T.
\end{eqnarray}

Define the {\it flag $h$-vector} and {\it flag $k$-vector} by the relations
$f_S = \sum_{T \subset S} h_T$ and $h_S = \sum_{T \subset S} k_T$.
The following is immediate from the definitions.

\begin{proposition}
For $F \in \QQ_{n+1}$, if
$F= \sum_{S\subset [n]} f_S\thinspace M_S$ then
\begin{eqnarray*}
F = \sum_{S\subset [n]} h_S\thinspace F_S
= \sum_{S\subset [n]} k_S\thinspace K_S.
\end{eqnarray*}
\label{basisprop}
\end{proposition}

Note that Proposition \ref{basisprop} holds, more specifically, for
a graded poset $P$ of rank $n+1$:
if $F(P)=\sum_{S\subset [n]} f_S(P) \thinspace M_S \in \QQ_{n+1}$
then $F(P) = \sum_{S\subset [n]} h_S(P)\thinspace F_S = 
\sum_{S\subset [n]} k_S(P)\thinspace K_S$, where
$f_S(P) = \sum_{T \subset S} h_T(P)$ and
$h_S(P) = \sum_{T \subset S} k_T(P)$.

\medskip
If $\I$ is a family of subsets of $[n]$, the we denote by $b[\I]$
the {\it blocking family} of $\I$, defined by
\begin{eqnarray*}
b[\I] = \{ S\subset [n] \ | \ S\cap I \not= \emptyset \hbox{\rm ~for all~}
I\in \I \}.
\end{eqnarray*}
We note that if $\I$ is an antichain in the Boolean lattice $2^{[n]}$, then
we can recover $\I$ as the set of minimal elements, under inclusion,
in $b[b[\I]]$.

We are particularly interested in the case in which the family $\I$ consists
of {\em intervals} in $[n]$, {\em i.e.}, subsets of the form
$\{ i, i+1, \dots, i+k\}$.  If $\I$ is such an interval family, we denote by
$F_\I$ the element of $\QQ_{n+1}$ defined by

\begin{eqnarray}
F_\I = \sum_{\stackrel{\scriptstyle S\subset [n]}{S\in b[\I]}} M_S.
\label{FI}
\end{eqnarray}

We call the $F_\I$ {\em interval quasisymmetric functions}.
For $S\subset [n]$, if we set $\I = \I(S) = \{\ \{i\} \ | \ i\in S \}$
then $F_{\I(S)} = F_S$.  We will see in the next section that the basis
for $\Pi$ can be represented in a similar manner.

In \cite{BH}, antichains of intervals
were used to describe the extreme rays of the closed convex cone
generated by all flag $f$-vectors of graded posets.  Equivalently,
the same description can be used to describe the closed convex cone
in $\QQ$ generated by all $F(P)$ arising from graded posets.  The following
is essentially \cite[Theorem 2.1]{BH}.
                                                                    
\begin{proposition}
The extreme rays of the closed convex cone in $\QQ_{n+1}$ generated by
all elements $F(P)$, where $P$ is a graded poset of rank $n+1$, are
precisely the interval quasisymmetric functions $F_\I$
corresponding to interval antichains in $[n]$.
\end{proposition}

Finally, we note that one can interpret the chain decompositions of
\cite{BH,BH2} as giving multiplication formulae for the $F_\I$.  In
particular, the proof of \cite[Proposition 2.8]{BH} yields the expression
\begin{eqnarray}
F(P)=\sum_C \sum_{S\in b[\I(C)]} M_S = \sum_C F_{\I(C)},
\label{chaindecomp}
\end{eqnarray}
where the first sum is over all maximal chains $C$ in $P$ and $\I(C)$ is
the interval antichain defined in \cite[p. 86]{BH}.  As in
\cite[Corollary 2.6]{BH}, we have $F_\I$ is the limit, as
$N \rightarrow \infty$, of elements
of the form $\frac{1}{f_{[n]}(P_N)}F(P_N)$, and one can use 
(\ref{chaindecomp}) to compute the representation of
$F_{\I_1} \cdot F_{\I_2}$ in terms of the $F_\I$.
                                                                                                                 
\section{The $\cd$-index and the peak algebra \label{cdpeak}}

Now suppose $P$ is an Eulerian poset of rank $n+1$ and
$F(P) = \sum_{S\subset [n]} f_S \thinspace M_S$.  We
wish to express $F(P)$ in terms of the basis $\{ \Theta_w\}$
for $\Pi_{n+1}$.  An unexpected outcome is that such a representation
is provided by the $\cd$-index of $P$.

\subsection{Blocking representations of $\Theta_w$\label{thetareps}}
We begin by giving a representation of the basis $\Theta_w$ in terms of
interval families associated with sparse subsets.

\begin{definition}
Let $S=\{ i_1,\dots,i_k\}\subset [n]$ and $w$ a $\cd$-word of degree $n$.
Then
\begin{enumerate}
\item
if $S$ is right sparse, let
\begin{eqnarray*}
\I_S = \{ \{i_1,i_1+1\},\{i_2,i_2+1\},\dots,\{i_k,i_k+1\}\},
\end{eqnarray*}
\item
if $S$ is left sparse, let
\begin{eqnarray*}
\I^S = \{ \{i_1-1,i_1\},\{i_2-1,i_2\},\dots,\{i_k-1,i_k\}\} \hbox{~ and}
\end{eqnarray*}
\item
$\I^w = \I^{S_w}$
\end{enumerate}
\end{definition}
When defined, both $\I_S$ and $\I^S$ are antichains of disjoint
two-element intervals in $[n]$.  The
interval antichains $\I_S$ and $\I^S$ are among what Bayer and Hetyei
refer to as {\it even interval systems} and so give rise (after their
doubling operation) to limits of flag $f$-vectors of Eulerian posets
\cite[Proposition 2.6]{BaH}.
We show that in this way $\I^w$ will give rise to $\Theta_w$. 

It is straightforward to see that for a degree $n$ $\cd$-word $w$ and
subset $S\subset [n]$, $S_w \subset S\cup(S+1)$ if and only if
$S\in b[\I^w]$.  Thus it follows from (\ref{theta}) and (\ref{thetaw})
that
\begin{eqnarray}
\Theta_w = \sum_{S\in b[\I^w]} 2^{|S|+1} M_S.
\label{thetab}
\end{eqnarray}
If we define the map $D:\QQ_{n+1} \longrightarrow \QQ_{n+1}$ by
$D(M_S) = 2^{|S|+1} M_S$, then (\ref{thetab}) is equivalent to
$$\Theta_w = D\left(F_{\I^w}\right),$$
where $F_{\I^w}$ is defined by (\ref{FI}).

It follows from \cite[Corollary 2.6]{BH} and the remark following
\cite[Definition 4]{BaH} that $\frac1 2 \Theta_w$ is the quasisymmetric
function corresponding to what Bayer and Hetyei call the doubled limit
poset $DP(n,\I^w)$.
From \cite[Theorem 4.2]{BaH} we obtain

\begin{proposition}
The $\Theta_w$ are among the extreme rays of the closed convex cone
in $\QQ$ generated by all $F(P)$ arising from Eulerian posets.
\label{extreme}
\end{proposition}

It will be helpful in what follows to have a representation of the
$\Theta_w$ in terms of the basis $\{ F_T \}$ of $\QQ$.  We let
$|w|_\d$ denote the \d-degree of the word $w$, {\it i.e.}, the number
of \d's in $w$.  The following is essentially \cite[Proposition 3.5]{Stem}.

\begin{proposition}
For any $\cd$-word $w$ of degree $n$,
\begin{eqnarray*}
\Theta_w = 2^{|w|_\d + 1} \sum_{T,\overline{T} \in b[\I^w]} F_T,
\end{eqnarray*}
where the sum is over all $T\subset [n]$, and $\overline{T} = [n] \setminus T$.
\label{thetaf}
\end{proposition}

\begin{proof}
By Proposition \ref{basisprop} and (\ref{thetab}),
\begin{eqnarray}
\Theta_w = \sum h_T F_T,
\label{hrep}
\end{eqnarray}
where the $h_T$ are defined uniquely by
\begin{eqnarray*}
\sum_{T\subset S} h_T = f_S =
\begin{cases}
2^{|S|+1} &S\in b[\I^w] \cr
0 & {\rm otherwise}.
\end{cases}
\end{eqnarray*}
Since $|w|_\d = |S_w|$, we need to show that
\begin{eqnarray}
h_T =
\begin{cases}
2^{|S_w|+1} &T,\overline{T}\in b[\I^w] \cr
0 & {\rm otherwise}.
\end{cases}
\label{htheta}
\end{eqnarray}

Assuming (\ref{htheta}), we compute
\begin{eqnarray}
\sum_{T\subset S} h_T = 2^{|S_w|+1} \cdot n_S^w,
\label{sumh}
\end{eqnarray}
where
\begin{eqnarray*}
n_S^w = \# \{ T\subset S \ | \ T,\overline{T} \in b[\I^w] \}. 
\end{eqnarray*}
If $S\notin b[\I^w]$, then $n_S^w=0$.  If $S\in b[\I^w]$, let
\begin{eqnarray*}
T_1 &=& \{ i\in S_w \ | \ i \in S, i-1 \notin S \}, \\
T_2 &=& \{ i\in S_w \ | \ i-1 \in S, i \notin S \}, \\
T_3 &=& \{ i\in S_w \ | \ \{i-1,i\} \subset S \}, \hbox{\rm ~and} \\
S' &=& S \setminus(T_1\cup(T_2-1)\cup T_3\cup(T_3-1)).
\end{eqnarray*}
We have $|T_1|+|T_2|+|T_3|=|S_w|$ and
$|S'| = |S|-|T_1|-|T_2|-2|T_3|$.
For a subset $T\subset S$, both $T$ and $\overline{T}$ are in $b[\I^w]$
if and only if 
$$T=[T_1\cup(T_2-1)]\cup R_3 \cup R_4,$$
where $R_3$ consists of one element from each pair $\{i-1,i\}$, $i\in T_3$
(these pairs are disjoint), and $R_4$ is {\it any} subset of $S'$.  Thus
\begin{eqnarray*}
n_S^w = 2^{|T_3|}\cdot 2^{|S|-|T_1|-|T_2|-2|T_3|} = 2^{|S|-|S_w|},
\end{eqnarray*}
and by (\ref{sumh})
\begin{eqnarray*}
\sum_{T\subset S} h_T = 2^{|S_w|+1} \cdot n_S^w = f_S,
\end{eqnarray*}
verifying (\ref{htheta}).
\end{proof}

We can now verify the form of the antipode of $\Pi$.

\begin{proof}[Proof of Proposition \ref{pianti}.]
It follows from \cite[Proposition 7.2]{Ehr} that if $s$ is the antipode on
$\QQ$, and $P$ is Eulerian, then
$$s\left(F(P)\right)= (-1)^{r(P)}F(P^*),$$
where $P^*$ is the dual or opposite or polar poset to $P$.
Thus the antipode of $\Pi$ is simply the
antipode $s$ restricted to $\Pi$.
Recall from \cite[Corollary 2.3]{MR}
that on $\QQ$, $s$ is given on the $F$ basis by
$$ s(F_T)= (-1) ^{n+1} F_{\overline{T}^\vee} $$
for $F_T=F_T^{(n+1)}\in \QQ_{n+1}$,
where, for $S\subset [n]$, $S^\vee = \{ n+1-i\ | \ i\in S \}$.  Therefore

\begin{eqnarray*}
s(\Theta _w) &=& 
s \left(2^{|w|_\d+1} \sum _{T, \overline{T} \in b[\I^w]} F_T \right) \cr
&=&(-1)^{n+1} 2^{|w|_\d+1} \sum _{T,  \overline{T} \in b[ \I^w]} F_{T^\vee}
\cr 
&=& (-1)^{\deg w +1}\left( 2^{|w|_\d+1} 
\sum _{T,\overline{T}\in b[(\I^w)^\vee]} F_{T} \right)
= (-1)^{\deg w+1}\Theta _{w^*},
\end{eqnarray*}
where $(\I^w)^\vee = \{ I^\vee \ |\ I\in \I^w\} = \I^{w^*}$.
\end{proof}

\subsection{$\Psi_w$ and the $\cd$-index}

For any Eulerian poset $P$ of rank $n+1$, there is a polynomial of degree
$n$, $\psi_P \in \Z\langle \c,\d \rangle$, called the $\cd$-{\it index} \cite{bk}.
(Here we assume $\deg \c =1$ and $\deg \d =2$.)  We denote
by $[w]$ or $[w]_P$ the coefficient of $w$ in $\psi_P$.  The coefficient
$[w]_P$ can be expressed linearly in terms of the {\it sparse flag $k$-vector},
that is, in terms of the numbers $k_S(P)$ for right sparse $S\subset [n]$.
See \cite[Proposition 7.1]{BiE} for this expression.  Of interest here is
the inversion of this relation \cite[Definition 6.5]{BER2}, which we write
as follows.

\begin{proposition}
For right sparse $S\subset[n]$,
\begin{eqnarray*}
k_S = \sum_{\stackrel{\scriptstyle S_w\in b[\I_S]}{|w|_\d=|S|}} [w].
\end{eqnarray*}
\label{ksw}
\end{proposition}

\begin{proof}
The expression in \cite[Definition 6.5]{BER2} sums over all $w$ of degree
$n$ that {\it cover} $S$ and have $|S|$ \d's.  Noting that in \cite{BER2},
the indexing is by dimension, not by rank as in this paper (and in \cite{BiE}),
it follows that $w$ covers $S$ if and only if $S\subset S_w\cup(S_w-1)$.
Since $|S|=|S_w|$, we can conclude $w$ covers $S$
if and only if $S_w\in b[\I_S]$.  
\end{proof}

\begin{remark}
We note that the relation in Proposition \ref{ksw} (more precisely, its
inverse \cite[Proposition 7.1]{BiE}) gives us a way to define a $\cd$-index
$\psi_F$ for any $F=\sum k_S K_S \in\Pi_{n+1}$ -- in fact, for any
$F\in \QQ_{n+1}$ -- by defining $[w]=[w]_F$, for
$\deg w = n$, directly from the coefficients $k_S$.  Further, for
non-homogeneous $F\in \QQ$, we can define $[w]_F$ for all $\cd$-words $w$
by $[w]_F = [w]_{F_i}$, where $F_i$ is the homogeneous component of $F$
of degree $\deg w + 1$.
\end{remark}

\smallskip
\begin{example}
For $F\in \QQ_3$, $F=k_\emptyset K_\emptyset + k_1\thinspace K_1 + 
k_2\thinspace K_2 +k_{12}\thinspace K_{12}$
and so we define $\psi_F= k_\emptyset\thinspace \c^2 + k_1\thinspace \d$.
For $F\in \QQ_4$, 
$F=\sum_{S\subset [3]} k_S\thinspace K_S$ and so
$$\psi_F= k_\emptyset\thinspace \c^3 + (k_2-k_1)\thinspace\c\d +
k_1\thinspace \d\c.$$
Note that in both cases, the values of $k_S$ for non-sparse $S$ are not
relevant to the definition of $\psi_F$.  For $F\in \Pi$, these values are
determined by the others.  For general $F\in\QQ$, this is no longer the case
since there are no relations on the $f_S$, and so on the $k_S$
\cite[Proposition 1.1]{BL}.
\label{cdexamp}
\end{example}
\smallskip

We now define another set of $a_{n+1}$ elements in $\QQ_{n+1}$ indexed by
words $w$ of degree $n$ and relate them to the $\Theta_w$.

\begin{definition}
For $w$ a $\cd$-word of degree $n$, let
\begin{eqnarray*}
\Psi_w = \sum_{\stackrel{\scriptstyle S\in b[\I^w]}{|S|=|w|_\d}}K_S,
\end{eqnarray*}
where the sum is over only right sparse $S\subset [n]$.
\label{psi}
\end{definition}

Consider the projection operator $F \mapsto \overline{F}$ on $\QQ$ defined
by
$$\overline{M}_S=
\begin{cases}
M_S & \hbox{\rm if}~ S ~\hbox{\rm is right sparse}\\
0 & \hbox{\rm if not}.
\end{cases}$$
Note that if $S$ is not right sparse then $\overline{K}_S=\overline{F}_S=0$.
This projection operator is injective on $\Pi$:
\begin{proposition}
If $F,G\in\Pi$ and $\overline{F}=\overline{G}$ then $F=G$.
\label{proj}
\end{proposition}

\begin{proof}
It is shown in \cite{BB} that a consequence of the generalized
Dehn-Sommerville relations is that the
flag $f$-vector for Eulerian posets is determined by its values
on right sparse subsets.  It follows from Proposition \ref{DSQ}
that this continues to hold for elements of $\Pi$.
\end{proof}

\begin{corollary}
For any $F\in \QQ$, there is a unique element $\pi(F)\in \Pi$ such that
$\overline{\pi(F)} = \overline{F}$.  The corresponding
map $\pi : \QQ\rightarrow \Pi$ is a linear projection.
\label{lift}
\end{corollary}

\begin{proof}
Again, from \cite{BB} we have that the right sparse subsets form a basis
for the flag $f$-vectors of Eulerian posets, and so for all of $\Pi$.
Given an $F\in \QQ$, the values of $f_S$ over all right
sparse subsets and the generalized Dehn-Sommerville equations determine
values for the remaining $f_S$ in such a way as to determine an element of
$\Pi$.  Call this element $\pi(F)$.  That $\pi(F)$ is unique follows from
Proposition \ref{proj}.

Note that if $F\in \Pi$, $\pi(F)=F$.  That the map $\pi$ is linear follows
from its construction.
\end{proof}

\begin{example}
For $F\in \Pi_3$, the generalized Dehn-Sommerville relations imply that
$f_2 = f_1$ and $f_{12} = 2\thinspace f_1$.  Thus for any
$F=f_\emptyset M_\emptyset + f_1\thinspace M_1 + 
f_2\thinspace M_2 +f_{12}\thinspace M_{12} \in \QQ_3$,
$$\pi(F) = f_\emptyset\thinspace M_\emptyset + f_1\thinspace
(M_1+M_2+2\thinspace M_{12}).$$
\end{example}
We call $\pi(F)\in \Pi$ the {\it Eulerian projection} of $F$.  That $\pi$
is not an algebra map can be seen from the fact that $\pi(M_1^{(2)})=0$
but $\pi( M_1^{(2)} \cdot M_1^{(2)})\not= 0$.  Note that for any
$F\in\QQ$, $[w]_F = [w]_{\pi(F)}$ and so the fibers of $\pi$ consist of $F\in\QQ$
having the same $\cd$-index.

The elements $\overline{\Psi}_w$ form a basis for the subspace
$\overline\QQ = {\rm span}\{M_S \ | \ S \hbox{\rm ~right sparse}\} \subset\QQ$.  
We see next that for $F \in \QQ$ the
coefficients of the expression of $\overline{F}$ in terms of this basis
are given by the $\cd$-index of $F$.

\begin{proposition}
For $F\in \QQ_{n+1}$,
\begin{eqnarray*}
\overline{F}= \sum_{\deg w=n} [w]\thinspace \overline{\Psi}_w,
\end{eqnarray*}
where $[w] = [w]_F$.
\label{projrep}
\end{proposition}

\begin{proof}
By Proposition \ref{basisprop}, we can write $F=\sum_S k_S\thinspace K_S$
and so
\begin{eqnarray}
\overline{F}=\sum_{S \hbox{\Small ~sparse}} k_S \thinspace \overline{K}_S = 
	\sum_{S \hbox{\Small ~sparse}} \left(
	\sum_{\stackrel{\scriptstyle w: |w|_\d = |S|}{S_w\in b[\I_S]}}[w]\right)\overline{K}_S,
\label{krep}
\end{eqnarray}
by Proposition \ref{ksw}, where the sum is over $w$ of degree $n$ and
$[w]=[w]_F$.   Here {\it sparse} means
right sparse.  When $|S|=|w|_\d=|S_w|$, we have $S_w\in b[\I_S]$
if and only if $S\in b[\I^w]$, so (\ref{krep}) becomes
\begin{eqnarray}
\overline{F} = \sum_{\deg w = n}[w] \left(
	\sum_{\stackrel{\scriptstyle S \mathrm{~sparse}}
  { \stackrel{\scriptstyle |S|=|w|_\d}{S \in b[\I^w]} }} \overline{K}_S \right)
= \sum_{\deg w = n}[w]\thinspace \overline{\Psi}_w.
\end{eqnarray}

\end{proof}

\subsection{$\Theta_w$ and the \ctwod-index}

We determine the relationship between $\Psi_w$ and $\Theta_w$ and
thereby a formula for the representation of $F \in \Pi$ in terms of
the $\Theta_w$.

\begin{proposition}
For any $\cd$-word $w$,
\begin{eqnarray*}
\overline{\Psi}_w = \frac {1}{2^{|w|_\d +1}}~~ \overline{\Theta}_w.
\end{eqnarray*}
\label{psitheta}
\end{proposition}

\begin{proof}
Suppose $w$ has degree $n\ge 0$.
By (\ref{thetab}) we have
\begin{eqnarray}
\overline{\Theta}_w = 2\sum_{\stackrel{\scriptstyle S \mathrm{~sparse}}
	 {S\in b[\I^w]}}  2^{|S|} M_S,
\label{thetabar}
\end{eqnarray}
where all $S\subset [n]$ and sparse means right sparse.

Using (\ref{KS}), we write
\begin{eqnarray}
\overline{\Psi}_w = \sum_{\stackrel{\scriptstyle S \mathrm{~sparse}}
	 {\stackrel{\scriptstyle |S|=|w|_\d} {S \in b[\I^w]} }}
 \overline{K}_S
=   \frac{1}{2^{|w|_\d}}
	\sum_{\stackrel{\scriptstyle S \mathrm{~sparse}}
	 {\stackrel{\scriptstyle |S|=|w|_\d}  {S \in b[\I^w]} }}\left(
	\sum_{S \subset R\subset [n]} 2^{|R|}\thinspace\overline{M}_R
	\right).
\label{psibar}
\end{eqnarray}
Now suppose $S\not= S'$ are both right sparse, $|S|=|S'|=|w|_\d$ and
$S,S' \in b[\I^w]$.  Then any $R\supset S \cup S'$ is not right sparse
and so $\overline{M}_R = 0$.  Thus, combining (\ref{thetabar}) and
 (\ref{psibar}) we get
\begin{eqnarray*}
\overline{\Psi}_w = \frac {1}{2^{|w|_\d}}
	\sum_{ \stackrel{\scriptstyle S \mathrm{~sparse}} {S \in b[\I^w]} }
	2^{|S|} \thinspace M_S = \frac {1}{2^{|w|_\d +1}}~~\overline{\Theta}_w.
\end{eqnarray*}
\end{proof}

Following \cite{BER}, for any $F\in \QQ$, we call the quantities
\begin{eqnarray}
[[w]] = \frac {1}{2^{|w|_\d}} ~ [w]
\label{c2d}
\end{eqnarray}
the coefficients of the \ctwod-{\it index} of $F$, where
$[w]=[w]_F$.  When $F\in\Pi$, these coefficients provide a
representation of $F$ in terms of the basis elements $\Theta_w$.

\begin{theorem}
For $F \in \Pi$
\begin{eqnarray*}
F = \frac 1 2 ~ \sum_w [[w]] ~ \Theta_w.
\end{eqnarray*}
\label{cdtheta}
\end{theorem}

\begin{proof}
We know from Proposition \ref{projrep} and Proposition \ref{psitheta} that
\begin{eqnarray*}
\overline{F} =  \sum_w [w] ~ \overline{\Psi}_w = 
 \frac 1 2 ~ \sum_w [[w]] ~ \overline{\Theta}_w.
\end{eqnarray*}
But if
\begin{eqnarray*}
F' = \frac 1 2 ~ \sum_w [[w]] ~ \Theta_w,
\end{eqnarray*}
then $\overline{F} = \overline{F'}$ and so $F=F'$ by Proposition \ref{proj}.
\end{proof}

We restate the theorem explicitly in the poset case as

\begin{corollary}
If $P$ is any Eulerian poset, then
\begin{eqnarray*}
F(P) = \sum_w \frac {1}{2^{|w|_\d + 1}}~ [w]_P  ~ \Theta_w,
\end{eqnarray*}
where the $[w]_P$ are the coefficients of the $\cd$-index of $P$.
\end{corollary}

In \cite{BER,BER2} it is shown that any zonotope $Z$ has an integral 
\ctwod-index, and so if $\F(Z)$ is the lattice of faces of $Z$, then
if we let $F(Z)=F(\F(Z))$, we have
\begin{corollary}
For a zonotope $Z$,
\begin{eqnarray*}
2 F(Z) = \sum_w [[w]] \Theta_w
\end{eqnarray*}
is in the $\Z$-span of the $\Theta_w$ in $\Pi$.
\end{corollary}
This remains true for the dual face lattice of any
oriented matroid \cite{BER}.  Since the $\cd$-indices of $P$ and $P^*$ are
related by 
\begin{eqnarray}
[w]_{P^*} = [w^*]_P
\label{cddual}
\end{eqnarray}
(see \cite[\S3]{bk}), we have
\begin{corollary}
If $P$ is the face lattice of
any hyperplane arrangement or, more generally, oriented matroid, the
quasisymmetric function $F(P)$ has a half-integral representation in terms
of the $\Theta_w$.
\end{corollary}

\smallskip
\begin{remark}
It follows from Theorem \ref{cdtheta} that one can also view the elements
$\Theta_w$ as the limit as $m \rightarrow \infty$ of 
$2\left(\frac{2}{m}\right)^{|w|_\d}~F(P_{w,m})$, where $P_{w,m}$ is the poset defined
in the proof of \cite[Proposition 1.2]{SCD}.  Another possible
approach to Theorem \ref{cdtheta} could be made via \cite[Proposition 2.9]{BaH}
(see Proposition \ref{extreme} and the comments preceding it).
\end{remark}

\section{The Stembridge map \label{sigmasect}}

We describe in this section an algebra map defined by Stembridge in
\cite[Theorem 3.1(c)]{Stem}.  It is most natural when viewing the
algebras $\QQ$ and $\Pi$ as arising from ordinary and enriched
$P$-partitions of labeled posets. In this case, for a given labeled
poset, the map sends the quasisymmetric function in $\QQ$ obtained via the
ordinary theory to that in $\Pi$ obtained via the enriched theory.
In the case of a labeled chain, it relates bases for these algebras
in a simple manner that will serve as our definition.

	For $S\subset [n]$, define
\begin{eqnarray}
\Lambda(S) = \{ i\in S \ | \ i\not= 1, i-1 \notin S\}.
\label{lambda}
\end{eqnarray}
For any $S$, $\Lambda(S)$ is clearly left sparse.  If one writes $S$ as a
unique union of minimally many intervals, $\Lambda(S)$ will consist
of the first element of each such interval, excluding the element 1.
For example, $\Lambda(\{1,2,3,5,8,9\})= \{ 5,8\}$.
If one thinks of $S$ as the descent set of some permutation $\pi$ in the
symmetric group $\calS_{n+1}$,
then $\Lambda(S)$ consists of those descents that are preceded by
ascents, that is, the {\it peaks} of $\pi$.

We define the map $\vartheta : \QQ \rightarrow \Pi$ by
$\vartheta(F_S) = \Theta_{\Lambda(S)}$ for any $S\subset [n]$, $n \ge 0$,
where $\Theta_S$ is labeled as in the original definition (\ref{theta}).
It is proved in \cite{Stem} that $\vartheta$ is an algebra map.  It arises
naturally as the map that associates the weight enumerator of all
$P$-partitions of a labeled poset with that of all enriched $P$-partitions
of the same poset.  See \cite{Stem} for details.

\subsection{A random walk on peak sets}

As a
linear map, the restriction $\vartheta : \Pi_{n+1} \longrightarrow \Pi_{n+1}$
can be written, for $S\subset [n]$,
\begin{eqnarray}
\vartheta(\Theta_S) = 2^{|S|+1} \sum_{T,\overline{T} \in b[\I^S]}
	\Theta_{\Lambda(T)},
\label{sigma}
\end{eqnarray}
using Proposition \ref{thetaf}.
Equivalently, we can write
\begin{eqnarray}
\vartheta(\Theta_w) = 2^{|w|_\d+1} \sum_u \eta_{u,w} \Theta_u,
\label{sigmaw}
\end{eqnarray}
where
\begin{eqnarray}
\eta_{u,w}= \#\{T\subset [n] ~|~ T, \overline{T} \in b[\I^w]; \Lambda(T)=S_u\}.
\label{eta}
\end{eqnarray}

If we let $w$ and $u$ be $\cd$-words of degree $n$,
$S_w=\{w_1<w_2<\cdots <w_l\}$, $S_u=\{u_1<u_2<\cdots <u_m\}$,
$u_0=0$ and $u_{m+1}=n+2$, then we have (assuming the empty
product to be 1)

\begin{proposition}
For $\cd$-words $w$ and $u$,
\begin{eqnarray*}
\eta_{u,w} =
\begin{cases}
0 & {\rm if}~ |S_w \cap (u_i,u_{i+1})|>1 {\rm ~for~some~} i,\cr
\prod (u_{i+1}-u_i-1)
& {\rm otherwise},
\end{cases}
\end{eqnarray*}
where the product is taken over all $i$, $0\le i\le m$, such that
$S_w\cap(u_i,u_{i+1})=\emptyset$.
\label{etaform}
\end{proposition}

\begin{proof}
Let $\mathcal{A}=\{T\subset [n]\; |\; T,\overline{T}\in b[\mathcal{I}^w];
\Lambda(T)=S_u\}$.  For disjoint intervals $I,J$ of natural numbers, we will
write $I\prec J$ whenever $x<y$ for all $x\in I$ and $y\in J$.
Consider the  partition $S_w\cup S_u=I_1\cup\cdots \cup I_{r}$ into
maximal intervals, where $I_1\prec\cdots \prec I_r$.
Since $S_w$ and $S_u$ are both sparse sets, consecutive
elements in each $I_i$ alternate
between the two sets.  It is then easy to see that
for all $T\in\mathcal{A},$ $T\cap I_i = S_u\cap I_i$ for every $i$.  
In particular, $T\cap I_i$ does not depend on $T,$ and so
we only need to consider the possible elements of $T$ outside of $S_w\cup S_u$.

Let $[n+1]\backslash (S_w\cup S_u)=J_1\cup\cdots\cup J_s$ be a partition into
maximal intervals.
An interval $J$ is said to have {\it type} $xy,$ where $x,y\in\{w,u\},$ if
$I_i\prec J\prec I_{i+1}$ for some $i,$ and 
the last element of $I_i$ is in $S_x$ and the first element of $I_{i+1}$ is
in $S_y$.  For the sake of the argument, if an interval has more than one
possible type, we always choose the unique type which favors $u$.
If $\{0\}\prec J\prec I_1$ then $J$  will be given type $ux,$ where $x$ depends 
on the first element of $I_1,$ and
if $I_r\prec J\prec \{n+2\},$ then $J$ will be given type
$yu,$ where $y$ depends on the last element in $I_r$. Every $J_i$ now has a
unique type.

The condition $|S_w\cap(u_i,u_{i+1})|>1$ for some $i$ is equivalent to the
existence
of some $J_k$ of
type $ww$. In this case, for any $T\in\mathcal{A},$ there exists some $w_j$ such
that
$\Lambda(T)\cap [w_j,w_{j+1}]=\emptyset$ 
and $T$ contains exactly one element in  each interval $[w_j-1,w_j]$ and
$[w_{j+1}-1,w_{j+1}]$. This is clearly
impossible, so in this case $\mathcal{A}=\emptyset$.

Suppose now that no $J_i$ has  type $ww$. It is straightforward to verify
that for any $T\in\mathcal{A},$
if $J_i$ has type $uw$ then $T\cap J_i=J_i$; if $J_i$ has type $wu$ then
$T\cap J_i=\emptyset$; and
if  $J_i$ has type $uu$ then $J_i=(u_j,u_{j+1})$ for some $u_j$, and
 $T\cap J_i=[u_j+1, u_j+t]$ for some $0\le t\le u_{j+1}-u_j-2$.
(We set $[u_j+1,u_j]=\emptyset$.)
Thus, every $T\in\mathcal{A}$ is determined only by $T\cap J_i$ for all
intervals $J_i$ of type $uu$.
These are precisely the intervals $(u_j,u_{j+1})$, $0\le j\le m$,
such that $S_w\cap (u_j,u_{j+1})=\emptyset$. This shows that our formula is an
upper bound for $\eta_{wu}$.

For the reverse inequality, suppose that $T\subset [n]$ satisfies 
$T\cap I_i=S_u\cap I_i$ for all $i$; $T\cap J_i=J_i$ for all $J_i$ of type $uw$;
 $T\cap J_i=\emptyset$ for
all $J_i$ of type $wu$; and for all $J_i=(u_j,u_{j+1})$ of type $uu$,
 there exists a $0\le t\le u_{j+1}-u_j-2$ such that $ T\cap J_i=[u_j+1,u_j+t]$.
We first show that $T,\overline{T}\in b[\mathcal{I}^w]$. This is trivial if 
$S_w=\emptyset,$ so let $w_j\in S_w,$ and let $I_k$ be
the interval containing $w_j$.
Suppose that $w_j\in S_u$. In this case $w_j\in T,$ 
and if $w_j-1\in I_k,$ then $w_j-1\not\in T$ because $T\cap I_k$ is sparse. If
$w_j-1\not\in I_k,$ then $w_j-1\in J_i$ for some $J_i$ of type $uu$ or $wu,$ and
so  $w_j-1\not\in T$.
Now suppose that $w_j\not\in S_u$. In this case $w_j\not\in T,$ and if
$w_j-1\in I_k,$ then $w_j-1$ is in $S_u$ and hence $T$.
If $w_j-1\not\in I_k$ then $w_j-1$ is in some $J_i$ of type $uw$
(no $J_i$ has type $ww$), which implies
$w_j-1\in T$. In either case, $|T\cap [w_j-1,w_j]|=1$.

It remains to prove that $\Lambda(T)=S_u$. One can use a similar argument
to show that if  $u_j\in S_u,$ 
then $u_j-1\not\in T$. Therefore, $\Lambda(T)\supset S_u$. Let $x\in\Lambda(T)$,
so that $x\in T$ and $x-1\not\in T$.
If $x\in I_i$ for some $i,$ then $x\in S_u$ since $T\cap I_i\subset S_u$. If
$x\in J_i$ for some $i,$ then $J_i$ must have
type $uw$ or $uu$. In both cases, $x-1\in T,$ a contradiction. This completes
the proof.
\end{proof}

\begin{corollary}
The transformation $\vartheta$ is indecomposable on $\Pi_{n+1}$.
\label{indec}
\end{corollary}

\begin{proof}
Let $\Gamma$ be the directed graph on the $\cd$-words of degree $n$ defined
by $(u,w)\in\Gamma$ whenever $\eta_{u,w}>0$.  Then $\c^n$ is a sink in
$\Gamma$, {\em i.e.}, every node has an arc pointing to $\c^n$.
Further $(u,w)\in \Gamma$ whenever $w$ is obtained from $u$ by
replacing a $\c^2$ by $\d$.  Thus every $w$ is in a directed cycle with
the word $\c^n$, and the assertion follows.
\end{proof}

As in the proof of Proposition \ref{thetaf}, we have that
$$\# \{ T\subset [n] \ | \ T,\overline{T} \in b[\I^S] \} = 2^{n-|S|},$$
so every column of the matrix representing $\vartheta$ has sum $2^{n+1}$.
Thus we can conclude that $2^{n+1}$ is an eigenvalue of the linear map
$\vartheta$.  Since $\vartheta$ is indecomposable, the Perron-Frobenius theory
of nonnegative matrices (see, {e.g.}, \cite[Chapter 16]{Bell})
implies that $2^{n+1}$
is the largest eigenvalue of $\vartheta$ on the finite-dimensional
space $\Pi_{n+1}$.  It has multiplicity 1; the corresponding eigenvector
$p_{n+1}$ is nonnegative up to scaling.  These assertions are, in fact, all
verified in the next subsection.
The coefficients of this eigenvector have a particularly interesting interpretation.

\begin{proposition}
The distribution of peak sets in the symmetric group
$S_{n+1}$ gives the nonnegative eigenvector $p_{n+1} \in \Pi_{n+1}$
of $\vartheta$ corresponding to the eigenvalue $2^{n+1}$.  That is, if
$$p_{n+1}= \sum_{S \subset [n] \hbox{\rm ~left sparse}} p_S ~ \Theta_S,$$
where $p_S$ is the number of permutations in $S_{n+1}$ with peak set $S$,
then $$\vartheta(p_{n+1}) = 2^{n+1} p_{n+1}.$$
\label{peakdistn}
\end{proposition}

\begin{proof}
From the interpretation of the multiplication of
the generators $\Theta_w$ in terms of shuffles of sequences with peak sets
given by $S_w$ \cite[(3.1)]{Stem}, it follows that
$p_{n+1}= (\Theta_{\bf 1})^{n+1}$,
where $\Theta_{\bf 1}$ is the unique generator in degree 1
corresponding to the empty $\cd$-word {\bf 1}.  That is,
$(\Theta_{\bf 1})^{n+1}$ gives the distribution of peak sets in $S_{n+1}$.
It is easy to check that $\vartheta(\Theta_{\bf 1})=2\Theta_{\bf 1}$, and
since $\vartheta$ is an algebra map, we have
$$\vartheta((\Theta_{\bf 1})^{n+1})=2^{n+1}(\Theta_{\bf 1})^{n+1}.$$
\end{proof}

See \cite[p.784]{Stem} for an expression for the coefficients of
$p_{n+1}$ in terms of peak sets of shifted standard Young tableaux.
In fact, $p_{n+1}$ is the unique nonnegative eigenvector of $\vartheta$, since
eigenvectors corresponding to any other eigenvalue must
have coefficients (in terms of the $\Theta_w$) that sum to 0.  This is so since the
vector of ones is an eigenvector for the transpose of the matrix of $\vartheta$,
and eigenvectors for a matrix and its transpose corresponding to distinct eigenvalues
must be orthogonal.

One way to interpret Proposition \ref{peakdistn} is that the operator
$\frac{1}{2^{n+1}}\vartheta$ defines a random walk on the family of left
sparse subsets of $[n]$ with stationary distribution given by the
probability distribution of peak sets in a random permutation in $S_{n+1}$.
We conjecture that this random walk is a specialization of a random
walk on $S_{n+1}$ with uniform stationary distribution.  We have checked
this through $S_4$; in fact, in each case it suffices to take
a specialization of a random walk on the braid arrangement
defined in \cite{BHR}.

We give a complete analysis of the spectrum of $\vartheta$ in the next subsection.
In particular, we show that the eigenvalues of
$\frac{1}{2^{n+1}}\vartheta$ on $\Pi_{n+1}$
are $(\frac 1 4)^k$, $0\le k \le \lfloor \frac n 2 \rfloor$.

\subsection{Diagonalization of $\vartheta$}

We describe further the spectrum of $\vartheta$ and give a complete set of
eigenvectors in $\Pi_{n+1}$ for each $n\ge 0$.  We have already observed that
$\Theta_{\bf 1}$ is the unique eigenvector in $\Pi_1$, with corresponding
eigenvalue $\lambda = 2$.  We construct the remaining eigenvectors from
$\Theta_{\bf 1}$ by means of two simple operations.

Define the map $L:\QQ\rightarrow \QQ$ by $L(M^{(n)}_S) = M^{(n+1)}_S$ for
any $S\subset [n-1]$.  We will show that $L^2=L\circ L$ commutes with $\vartheta$,
and so $L^2$ preserves eigenvectors of $\vartheta$:
if $\vartheta(v)=\lambda~v$ then
$\vartheta(L^2(v)) = L^2(\vartheta(v))=\lambda~L^2(v)$,
showing $L^2(v)$ to be an eigenvector for the same eigenvalue.

Since $\vartheta$ is an algebra map, products of eigenvectors in $\Pi$ are
again eigenvectors.  In particular, if $\vartheta(v)=\lambda~v$ then
$\Theta_{\bf 1}\cdot v$ is an eigenvector for eigenvalue $2\lambda$.  We will
consider multiplication by  $\Theta_{\bf 1}$ as a linear map on $\Pi$,
also denoted as $\Theta_{\bf 1}$ when there is no possibility of confusion.

For any $\cd$-word $w=w(\c,\d)$, define the operator
$$\widehat{w} :\Pi \longrightarrow \Pi$$
by $\widehat{w}=w(\Theta_{\bf 1},L^2)$.  For example,
$\widehat{\c\d\c}=  \Theta_{\bf 1}\circ L^2 \circ \Theta_{\bf 1}$.
Note that if $w={\bf 1}$ then $\widehat{w}$ is the identity map.
It follows from the discussion
above that $\widehat{w}$ preserves eigenvectors of $\vartheta$ on $\Pi$,
multiplying the corresponding eigenvalue by the factor $2^{|w|_\c}$, where
$|w|_\c$ is the number of \c's in $w$.

The main result of this section is

\begin{theorem}
The map $\vartheta$ is diagonalizable on $\Pi$.  A complete set of eigenvectors
is given by 
$$\Omega_w = \widehat{w}(\Theta_{\bf 1}),$$
where $w$ is any $\cd$-word $w$, $\widehat{w}(\Theta_{\bf 1})$ is
the image of $\Theta_{\bf 1}$ under the map $\widehat{w}$.
The eigenvalue corresponding to
$\Omega_w$ is $2^{|w|_\c+1}$, and so, on $\Pi_{n+1}$, the eigenvalues of $\vartheta$
are $2^{n+1-2k}$, $0\le k \le \lfloor \frac n 2 \rfloor$.
\label{diagthm}
\end{theorem}

The proof of this result proceeds by a sequence of propositions.  The first of
these is the commutativity of $\vartheta$ and $L^2$ on $\QQ$.

\begin{proposition}
As maps on $\QQ$, $\vartheta \circ L^2 = L^2 \circ \vartheta$, so $L^2$ preserves
eigenvectors of $\vartheta$, as well as their eigenvalues.
\label{commprop}
\end{proposition}

\begin{proof}
It is straightforward to verify that for $S\subset [n-1]$
\begin{eqnarray*}
L(F^{(n)}_S) = F^{(n+1)}_S - F^{(n+1)}_{S\cup\{ n \}},
\end{eqnarray*}
and so
\begin{eqnarray}
L^2(F^{(n)}_S) = F^{(n+2)}_S - F^{(n+2)}_{S\cup\{ n \}} 
		- F^{(n+2)}_{S\cup\{ n+1\}} + F^{(n+2)}_{S\cup\{ n,n+1 \}}.
\label{L2F}
\end{eqnarray}
Similarly, we have
\begin{eqnarray}
L^2(\Theta_w) = \Theta_{w\c^2} - \Theta_{w\d}
\label{L2Theta}
\end{eqnarray}
or, equivalently,
$L^2(\Theta^{(n)}_T) = \Theta^{(n+2)}_{T} - \Theta^{(n+2)}_{T\cup\{n+1\}}$
for left sparse $T\subset [n-1]$.  Now, using (\ref{L2F}) and (\ref{L2Theta}),
one can verify that
\begin{eqnarray}
\vartheta \circ L^2\thinspace(F^{(n)}_S) = L^2 \circ \vartheta\thinspace(F^{(n)}_S)
	= \Theta^{(n+2)}_{\Lambda(S)} - \Theta^{(n+2)}_{\Lambda(S)\cup\{n+1\}}.
\end{eqnarray}
\end{proof}

We note that $\vartheta \circ L \not= L \circ \vartheta$; in particular, we have
$L\circ \vartheta (M^{(1)}_\emptyset) = 2M^{(2)}_\emptyset$ while 
$\vartheta \circ L (M^{(1)}_\emptyset) = 0$.
Next, we need to show that the eigenspaces induced by $L^2$ are independent of
those induced by $\Theta_{\bf 1}$.  This will follow from

\begin{proposition}
For each $n\ge 0$,
$$\QQ_{n+1} = L(\QQ_n) \oplus \Theta_{\bf 1}(\QQ_n).$$
\label{Qdirsum}
\end{proposition}

\begin{proof}
Since both $L$ and $\Theta_{\bf 1}$ are injective ($\QQ$ has no zerodivisors),
it is enough to prove $L(\QQ_n) \cap \Theta_{\bf 1}(\QQ_n)= \{0 \}$.  To this
end, recall that $\Theta_{\bf 1} = 2 M_{(1)}$, where $(1)$ is the unique
composition of $1$ (see (\ref{mbeta}) and (\ref{theta})).
Using the formula \cite[Lemma 3.3]{Ehr}
for multiplication in the basis $M_\beta$, we have
\begin{eqnarray*}
M_{(1)} \cdot M_\beta &=& M_{\beta,1} 
  + \sum_{i=1}^k \left(M_{(\beta_1,\dots,\beta_{i-1},1,\beta_i,\dots,\beta_k)}
  + M_{(\beta_1,\dots,\beta_{i-1},\beta_i+1,\beta_{i+1},\dots,\beta_k)}\right),
\label{Mbetamult}
\end{eqnarray*}
where $\beta = (\beta_1,\dots,\beta_k)$ is any composition of $n$.
Order compositions of $n+1$ first by the number of parts (those with fewer
parts are smaller in the order) then lexicographically, {\em i.e.},
$$\beta=(\beta_1,\dots,\beta_k) \prec (\beta'_1,\dots,\beta'_{k'})=\beta'$$
if $k<k'$ or $k=k'$ and for some $i$, $\beta_j =\beta'_j$ for $j<i$ while
$\beta_i < \beta'_i$.  With this order, the composition $(\beta,1)$ is the
largest index in the right-hand side of the expression for
$M_{(1)}\cdot M_\beta$ given above.  Since any element of
$L(\QQ_n)$ involves only combinations of
$M_{(\gamma_1,\dots,\gamma_l)}$, where $\gamma_l > 1$, this shows that
$L(\QQ_n) \cap \Theta_{\bf 1}(\QQ_n)= \{0 \}$.  
\end{proof}

From this we can conclude immediately
\begin{corollary}
For each $n\ge 1$,
$$\Pi_{n+2} = L^2(\Pi_n) \oplus \Theta_{\bf 1}(\Pi_{n+1}).$$
\label{Pidirsum}
\end{corollary}

Now we can complete the

\begin{proof}[Proof of Theorem \ref{diagthm}]
A basis of eigenvectors is constructed inductively, beginning with
$\Theta_{\bf 1}$ for $\Pi_1$ and $(\Theta_{\bf 1})^2$ for $\Pi_2$.
If we have constructed a basis for $\Pi_n$ and $\Pi_{n+1}$, then
applying $L^2$ to the former and $\Theta_{\bf 1}$ to the latter yields
a basis for $\Pi_{n+2}$ by Corollary \ref{Pidirsum}.  The resulting
basis consists of all $\Omega_w$, where $\deg w = n+1$.  The eigenvalue
corresponding to $\Omega_{\c^{n+1}}= \Theta_{\bf 1}^{n+2}$ is $2^{n+2}$
by Proposition \ref{peakdistn}.
Every substitution of a \d\  for a $\c^2$ divides the eigenvalue by 4.
\end{proof}

\begin{remark}Note that $\Omega_{\c^n}= \Theta_{\bf 1}^{n+1}$ is the peak set
distribution of $S_{n+1}$ as described in Proposition \ref{peakdistn}.
It would be interesting to see whether the other eigenvectors $\Omega_w$ have
similar combinatorial interpretations.
\end{remark}
\begin{remark} In (\ref{L2Theta}) we observe
$L^2(\Theta _w) = \Theta _{w\c ^2} - \Theta _{w\d}$. Similarly, it is
straightforward to observe
\begin{eqnarray*}
\Theta _1 (\Theta _w)&=& \Theta _{\c w} + \Theta _{w\c } + 
\sum _{w=w_1\c w_2} \Theta _{w_1\d w_2}\\
&&+\sum _{w=w_1\d w_2} (\Theta _{w_1\c\d w_2} +\Theta _{w_1\d\c w_2}).
\end{eqnarray*}
For example, 
\begin{eqnarray*}
\Theta _1(\Theta _{\c\d})&=&\Theta _{\c\c\d} + \Theta _{\c\d\c} + 
\Theta _{\d\d} + \Theta _{\c\c\d} + \Theta _{\c\d\c}\\
&=& 2 \Theta _{\c ^2\d} + 2 \Theta _{\c\d\c} + \Theta _{\d ^2}.
\end{eqnarray*}
\end{remark}
\begin{remark}
With the basis $\Omega_w$, we can define a new $\cd$-index for elements
$F\in \Pi$ or for Eulerian posets $P$, in which the
coefficient of the word $w$ is given by the corresponding coefficient of
the basis element $\Omega_w$ in the expression of $F$ or $F(P)$.  This
does not appear to have reasonable properties for face posets of polytopes,
although it is nonnegative for {\em simplicial} 3-polytopes.
\end{remark}
\begin{remark}
The cone in $\Pi_{n+1}$ spanned by all $\Omega_w$, $\deg w = n$ is not
invariant under the antipode $s$ on $\Pi$, as is that spanned by the $\Theta_w$.
On the other hand, its extreme rays, and so all its faces, are fixed by
the combinatorially interesting map $\vartheta$.  It might be useful
to have a basis invariant under both $s$ and $\vartheta$.  The corresponding
index might have some interesting properties.
\end{remark}

\subsection{Peaks, hyperplane arrangements and Gorenstein$^*$ posets}

It has been pointed out to us by Aguiar and Bergeron (personal communications)
that the map $\vartheta$ is essentially the map $\omega$ of \cite{BER}.  More
precisely, if $L$ is any geometric lattice, let $L_{\hat{0}}$ be the lattice $L$
with a {\it new} minimal element $\hat{0}$ added.  Then $L_{\hat{0}}$ is a
graded lattice and so $F(L_{\hat{0}}) \in \QQ$.

\begin{proposition}
For the geometric lattice $L$ of an oriented matroid ${\mathcal O}$,
$$\vartheta(F(L_{\hat{0}}))=2 \ F(Z),$$
where $Z$ is the dual face lattice of ${\mathcal O}$.  In particular, when
${\mathcal O}$ corresponds to an arrangement of hyperplanes, then $Z$
is the face lattice of the associated zonotope.
\label{zasprop}
\end{proposition}

\begin{proof}
If we give the usual $R$-labeling to $L$, and label the unique cover
relation over $\hat{0}$ by 0, then this follows from the observation
of Aguiar-Bergeron and \cite[Corollary 3.2]{BER}.
\end{proof}

One can view Proposition \ref{zasprop} as a complete summary of the
relationship between enumerative invariants of chains in a central hyperplane
arrangement and those of the associated lattice of intersections, whose
study was begun by Zaslavsky in \cite{Zas}.

Since geometric lattices are known to be Cohen-Macaulay posets,
that is, the associated complex of chains is a Cohen-Macaulay complex
\cite{Stanb}, it follows that $L_{\hat{0}}$ is also Cohen-Macaulay and so
$F(L_{\hat{0}})$ has a nonnegative representation in the basis $\{F_S\}$
of $\QQ$.  As a consequence, we get from Proposition \ref{zasprop} a
special case of \cite[Corollary 2.2]{SCD}, namely, we can conclude
that arrangements and zonotopes have nonnegative $\cd$-indices.

A poset is called {\em Gorenstein$^*$} if it is both Eulerian and Cohen-Macaulay.
Such posets include all face posets of spherical complexes.  Stanley
has conjectured that if $P$ is Gorenstein$^*$, then it has a nonnegative $\cd$-index,
that is, $[w]_P\ge 0$, for all $\cd$-words $w$ \cite[Conjecture 2.1]{SCD}.
In light of Theorem \ref{cdtheta}, this amounts to saying that for $P$
Gorenstein$^*$, $F(P)$ must lie in the cone in $\Pi_{n+1}$ generated by the
$\Theta_w$, $\deg w = n$, that is, the nonnegative orthant of $\Pi_{n+1}$
defined by the basis $\{\Theta_w\}$.

The map $\vartheta$ allows us to define a slightly larger simplicial cone than the
nonnegative orthant in $\Pi_{n+1}$ that must contain $F(P)$ for Gorenstein$^*$
posets $P$.

\begin{proposition}
For Cohen-Macaulay posets $P$, we always have $\vartheta(F(P)) \ge 0$, that is,
$\vartheta(F(P))$ always lies in the cone in $\Pi_{n+1}$ generated by the
$\Theta_w$, $\deg w = n$.
\label{G*}
\end{proposition}

\begin{proof}
By Proposition \ref{basisprop}, we have $F(P) = \sum h_S F_S$, where $h_S\ge 0$
since $P$ is Cohen-Macaulay \cite{Stanb}.  The proposition now follows from the
definition of $\vartheta$.
\end{proof}

Considering $\vartheta$ restricted to $\Pi_{n+1}$, we can view the set
$\{ \ F\in \Pi_{n+1}\ | \ \vartheta(F)\ge 0 \ \}$ as a simplicial cone in
$\Pi_{n+1}$.  A more explicit description in terms of inequalities on the
coefficients $[w]_P$ is given by the rows of the matrix $(\eta_{u,w})$
in (\ref{eta}).
This cone includes the image of the nonnegative orthant in $h$-space
under the linear map that takes the flag-$h$ vector to the $\cd$-index.
That this latter cone is given by the inequalities
\begin{eqnarray}
h_T = \sum_{T,\overline{T}\in b[\I^w]} [w] \ge 0
\label{hnonneg}
\end{eqnarray}
follows directly from \cite[Proposition 1.3]{SCD} or from Proposition \ref{thetaf}.
It is straighforward to obtain the inequalities in Proposition \ref{G*} from those
in (\ref{hnonneg}): to get the inequality given by row $u$ in $(\eta_{u,w})$,
add the expression for $h_T$ over all $T$ for which $\Lambda(T) = S_u$.

\smallskip
\begin{example}
If $P$ is Gorenstein$^*$  and the rank of $P$ is 4, then the cone described in
Proposition \ref{G*} is given in $\cd$-coordinates by the inequalities
\begin{eqnarray*}
4[\c^3] + \phantom{2}[\c\d] +\phantom{2} [\d\c] &\ge& 0 \cr
2[\c^3] + 2[\c\d] +\phantom{2} [\d\c] &\ge& 0 \cr
2[\c^3] + \phantom{2}[\c\d] +2 [\d\c] &\ge& 0.
\end{eqnarray*}
On the other hand, the nonnegativity of the $h_S$ imply directly that
$$\begin{array}{cccccccccc}
h_2&=&h_{13}&=& [\c^3]& +& [\c\d]& +& [\d\c]  &\ge 0 \cr
h_1&=&h_{23}&=& [\c^3] &&& +&[\d\c]  &\ge 0 \cr
h_3&=&h_{12}&=& [\c^3]& +&[\c\d]&&  &\ge 0 \cr
h_\emptyset&=&h_{123}&=& [\c^3]&&&&   &\ge 0 \cr
\end{array}$$
The second system clearly implies the first.
\end{example}

\section{The $g$-homomorphism\label{gsect}}

We define an algebra homomorphism from $\QQ$ to $\Q[x]$ that extends the
definition of the $g$-polynomial of a graded poset.
In the case of the face lattices of (rational) convex polytopes, this
polynomial is related to the Poincar\'e polynomial of the associated toric
variety.  For all rational polytopes, the $g$-polynomial is known to have
nonngegative coefficients; in the case of simplicial convex
polytopes, this fact is known as the {\it generalized lower bound theorem}.
It was proved by Stanley \cite{SGLB,s-gh} by means
of the toric variety associated to a rational polytope.
It remains open for nonrational polytopes.

We begin by defining the $g$-polynomial of a graded poset.  For any graded
poset $P$ of rank $n+1$ we define two polynomials $f(P,x),g(P,x)\in\Q[x]$
(actually in $\Z[x]$) recursively as follows.
If $n+1=0$, then $f(P,x)=g(P,x)=1$.  If $n+1>0$, then
\begin{eqnarray}
f(P,x)= \sum_{y\in P\setminus \{\hat{1}\}} g([\hat{0},y],x)(x-1)^{n-r(y)}.
\label{fdef}
\end{eqnarray}
If $f(P,x)=\sum_{i=0}^{n} \kappa_i x^i$ has been defined, then we define
\begin{eqnarray}
g(P,x)= \kappa_0 + \left(\kappa_1-\kappa_0 \right)\thinspace x+ \cdots +
\left(\kappa_{\lfloor\frac{n}{2}\rfloor} -
 \kappa_{\lfloor\frac{n}{2}\rfloor -1}\right)\thinspace
 x^{\lfloor\frac{n}{2}\rfloor}.
\label{gdef}
\end{eqnarray}
For an Eulerian poset $P$, the vector
$(h_0,\dots, h_n)=(\kappa_n,\dots,\kappa_1,\kappa_0)$ is what is
usually called the {\em toric $h$-vector} of $P$.  Since for Eulerian
$P$, $h_i=h_{n-i}$ \cite{s-gh}, our definition of $g(P,x)$ agrees with
the usual one in the Eulerian case.  We note that in \cite{BE}, this
distinction between $\kappa_i$ and $h_i$ is not made, so their formulas for
$h_i$ are, in reality, for $h_{n-i}$.

Since the coefficients of $g(P,x)$ are integer linear combinations of
the quantities $f_S(P)$ (see, for example, \cite[Theorem 6]{bk},
\cite[Theorem 3.1]{BE} or \cite[\S 4.3]{BL}), these necessarily unique
expressions can be used to extend this definition to give a linear map
\begin{eqnarray}
g: \QQ \longrightarrow \Q[x],
\label{gQ}
\end{eqnarray}
satisfying $g(F(P))=g(P,x)$ for any graded poset $P$.  That $g$ is an
algebra homomorphism follows from the following observation, which was
first noted in \cite{K} in the case of polytope face lattices.  Its
proof depends on the fact that an interval in a product of posets is the
product of intervals from each, and seems not to have appeared in this
generality anywhere.

\begin{proposition}
For graded posets $P$ and $Q$, 
$$g(P\times Q,x)= g(P,x) g(Q,x).$$
\label{gproduct}
\end{proposition}

\begin{proof}
The conclusion is immediate if $r(P\times Q)=0$.
Otherwise, using (\ref{fdef}) and induction, we get
\begin{eqnarray*}
(1-x)f(P\times Q,x) &=& g(P,x)\thinspace (1-x)f(Q,x) +
	(1-x)f(P,x)\thinspace g(Q,x) \cr
 & & \quad -~ (1-x)f(P,x)\thinspace (1-x)f(Q,x).
\end{eqnarray*}
By (\ref{gdef}), $g(P \times Q,x)$ consists of the terms of $(1-x)f(P\times Q,x)$
of degree at most $(r(P)+r(Q)-1)/2$.
Writing $(1-x)f(P,x)=g(P,x)+\tilde{g}(P,x)$, similarly
for $Q$, we note that all the terms of $\tilde{g}(P,x)$ (respectively,
$\tilde{g}(Q,x)$) have
degree at least $r(P)/2$ (respectively, $r(Q)/2$).  Now
\begin{eqnarray*}
(1-x)f(P\times Q,x) = g(P,x)\thinspace g(Q,x) -
 \tilde{g}(P,x)\thinspace \tilde{g}(Q,x),
\end{eqnarray*}
where the last term has only terms of degree at least $(r(P)+r(Q))/2$.  The
proposition follows.
\end{proof}

Using the fact that $\QQ$ is spanned by elements of the form $F(P)$
\cite[Proposition 1.1]{BL}, and
recalling that $F(P\times Q)=F(P)F(Q)$ \cite{Ehr}, we can conclude

\begin{corollary}
The map
$$g: \QQ \longrightarrow \Q[x]$$
is an algebra homomorphism.
\label{ghomo}
\end{corollary}

\begin{proof}
We need only check multiplicativity.
Suppose $G,H\in \QQ$, $G=\sum_i \alpha_i \thinspace F(P_i)$ and
$H=\sum_j \beta_j \thinspace F(Q_j)$.  Then
\begin{eqnarray*}
g(GH) &=& \sum_{i,j} \alpha_i\beta_j \thinspace g\left(F(P_i)F(Q_j)\right) \cr
	&=& \sum_{i,j} \alpha_i\beta_j \thinspace g\left(F(P_i\times Q_j)\right) \cr
	&=& \sum_{i,j} \alpha_i\beta_j \thinspace g(P_i,x)g(Q_j,x) \cr
	&=& g(G)g(H),
\end{eqnarray*}
by Proposition \ref{gproduct} and the fact that $g(F(P))=g(P,x)$.
\end{proof}

Restricted to $\Pi$, there is an explicit formula for $g$, due
essentially to Bayer and Ehrenborg \cite{BE}.  We
follow the development in \cite{BE} to express this.  Define
$p(n,k) = \binom{n}{k} - \binom{n}{k-1}$ and polynomials
$$Q_{n+1} = \sum_{k=0}^{\lfloor\frac{n}{2}\rfloor} (-1)^k p(n,k)x^k$$
for any $n$ and 
$$T_{n+1} = (-1)^{\frac n 2} p\left(n,\frac n 2\right)x^{\frac n 2}$$
for $n$ even.  Note that $Q_1=T_1=1$.

Say that a $\cd$-word $w$ is {\em even} if every element of
$S_w$ is even, that is, if 
$w = \c^{n_1}\d\c^{n_2}\d\cdots\c^{n_k}\d\c^{m}$, and
$n_1,\dots,n_k$ are all even.  The following is an interpretation of
\cite[Theorem 4.2]{BE} in our context.  It follows since $\Pi$ is spanned by
elements of the form $F(P)$, where $P$ is Eulerian.

\begin{proposition}
If $w = \c^{n_1}\d\c^{n_2}\d\cdots\c^{n_k}\d\c^{m}$, then
\begin{eqnarray*}
g(\Theta_w) =
\begin{cases}
2^{k+1}~x^k~Q_{m+1}~\prod\limits_{j=1}^k T_{n_j+1} & {\rm if~} w {\rm ~is~even,}\cr
0 & {\rm otherwise.}
\end{cases}
\end{eqnarray*}
\label{gcalc}
\end{proposition}

\begin{remark}
Note that $g(\Theta_w)$ depends only on the initial and inter-peak distances
of the peak set indicated by $w$, but not on their order, vanishing when
any one of these is odd.  One could easily describe the kernel of the
$g$ map from this.  That $g$ is multiplicative on $\Pi$ is not evident
from the expression in Proposition \ref{gcalc}.
\end{remark}

\begin{remark}
Since the basis $\Omega_w$ is partially multiplicative, the images
$g(\Omega_w)$ should have a simpler expression than that of Proposition
\ref{gcalc}.  In particular, since $g(\Theta_{\bf 1})=1$, the calculation
of $g(\Omega_w)$ is determined entirely by the effect of the map $L^2$.
\end{remark}

\providecommand{\bysame}{\leavevmode\hbox to3em{\hrulefill}\thinspace}

\end{document}